\documentclass[12pt]{article}

\usepackage{amsmath}
\usepackage{amsthm}

\newtheorem{theorem}{Theorem}[section]

\newtheorem{conjecture}[theorem]{Conjecture}

\theoremstyle{remark}

\newtheorem*{sketch of proof}{Idea of Proof}

\theoremstyle{definition}

\newcommand{\script}[1]{\text{$\cal{#1}$}}
\newcommand{\comment}[1]{}

\usepackage[margin=1in]{geometry}

\begin{document}

\title{Elementary Techniques for Erd\H os--Ko--Rado-like Theorems}
\author{Greg Brockman and Bill Kay\\
\small Harvard University, University of South Carolina}
\date{\today}

\maketitle

\begin{abstract}
The well-known Erd\H os--Ko--Rado Theorem states that if $\script F$ is a family of $k$-element subsets of $\{1,2,\ldots,n\}$ ($n\geq 2k$) satisfying $S,T\in\script F\Rightarrow|S\cap T|\geq 1$, then $|\script F|\leq\binom{n-1}{k-1}$.  The theorem also provides necessary and sufficient conditions for attaining the maximum.  We present elementary methods for deriving generalizations of the Erd\H os--Ko--Rado Theorem on several classes of combinatorial objects.  We also extend our results to systems under Hamming intersection.
\end{abstract}

\section{Introduction}

As the reader knows, mathematics is many things, but it is never a one-stop trip.  The development of a mathematical theory sees shaping by many hands, beginning with the innovator who first dreamed up its foundations and ending years down the road, if at all.  Often, the discovery of one fundamental theorem will tip over the domino of another result in an observer's mind, prompting new results and raising new questions.  These results can then prove to be a catalyst of innovation in the work of others, leading to the development of a broad-reaching mathematical monolith.

The Erd\H os--Ko--Rado Theorem is such a fundamental result.

This discovery was announced to the world in 1961 with the publication of the seminal paper \textit{Intersection theorems for systems of finite sets} by Paul Erd\H os, Chao Ko, and Richard Rado~\cite{ekr}.  Interestingly enough, they had the paper essentially completed as early as 1938; however, one reason they did not publish it was due to what Erd\H os~\cite{workwithrado} describes as ``relatively little interest in combinatorics'' at the time.  Ironically, the Erd\H os--Ko--Rado paper has since become perhaps the most oft-cited of Erd\H os's joint work with Rado.

So what is this celebrated theorem?  Let $[n] := \{1,2,\ldots, n\}$, and consider any family $\script F$ of $k$-sets on $[n]$ (by this, we mean $k$-element subsets of $[n]$).  If we are told that $\script F$ is \textit{intersecting}, that is, every pair of $k$-sets $S,T\in\script F$ satisfy $|S\cap T|\geq 1$, what is the maximum possible size of our family $\script F$?  Note that if $n < 2k$, by the Pigeonhole Principle \textit{every} pair of $k$-sets has a nonempty intersection, and hence any family $\script F$ of these sets will be intersecting---not a hugely subtle case.  However, if we restrict $n\geq 2k$, things start getting interesting.

In this case, there is a plethora of non-intersecting families of $k$-sets.  However, a bit of inspection reveals that there is an easy-to-define intersecting family that is quite populous: in particular, take an arbitrary element of $[n]$ (say, 1) and consider the family of all sets containing that element.  Such a family we will dub \textit{trivially intersecting}, or just \text{trivial} for short.  It is not hard to count that all trivial families have $\binom{n-1}{k-1}$ elements.  But as always in mathematics, we are forced to ponder, can we do any better than this, given sufficient cleverness?  How often is it, really, that the obvious way turns out to be the best one?

And this is where the Erd\H os--Ko--Rado Theorem comes into play: the theorem's statement is that indeed the trivial family is actually (strictly) the best!  (Incidentally, there is a mnemonic to describe situations such as this: ``TONCAS,'' or ``The Obvious Necessary Conditions are Also Sufficient.''  We borrow this term from West ~\cite{introductiontographtheory}, who attributes it to Nash-Williams and others.)  The original method for proving the Erd\H os--Ko--Rado Theorem was two pages worth of induction, casework, and contradiction.  However, in 1972, Katona~\cite{katona} published a quite short and strikingly beautiful probabilistic proof, which we have adapted below.

\begin{theorem}[Erd\H os--Ko--Rado, henceforth EKR] Let $\script F$ be an intersecting family of $k$-sets on $[n]$, $n\geq 2k$.  Then
$$\max |\script F| \leq \binom{n-1}{k-1},$$
with equality if and only if $\script F$ is trivial.
\end{theorem}

\begin{proof}
Let $\sigma = \sigma(1)\sigma(2)\ldots \sigma(n)$ be a permutation of $[n]$.  We'll say that $\sigma$ \textit{contains} a $k$-set $S\in\script F$ if all of $S$'s elements appear in consecutive positions in $\sigma$ when $\sigma$ is read cyclically.  For example, if $n=7,k=3$, and $\sigma=(5,2,1,7,4,6,3)$, then $\sigma$ would contain $\{1,2,5\}$ and $\{3,5,6\}$, among others (but not, say, $\{5,7,6\}$).  If $S$ is contained in $\sigma$, define the \textit{head} of $S$ to be the first element of $S$ to appear in $\sigma$.

Notice that $\sigma$ can contain at most $k$ elements of $\script F$: let $i$ be the least number such that $\sigma(i)$ is the head of one of $\script F$'s $k$-sets that is contained in $\sigma$.  Then the only other possible heads of these $k$-sets are $\sigma(i+1),\sigma(i+2),\ldots,\sigma(i+k-1)$, since $\script F$'s sets all intersect the $k$-set $\{\sigma(i),\sigma(i+1),\ldots,\sigma(i+k-1)\}$.

On the other hand, every $k$-set in $\script F$ is contained in exactly $k!\cdot n\cdot (n-k)!$ ($k!$ orderings of the set, $n$ choices of which position in $\sigma$ to place the head, and $(n-k)!$ orderings of the other elements in $\sigma$) permutations of $[n]$.  Since there are $n!$ total permutations, we combine all of this to obtain
$$|\script F| k!\cdot n\cdot (n-k)! \leq n!\cdot k,$$
which rearranges to the desired
$$\script F \leq \frac{(n-1)!}{(k-1)!(n-k)!} = \binom{n-1}{k-1}.$$
We leave to the reader the proof of the equality case.
\end{proof}

In the spirit of Katona's proof, we seek in this paper to provide elementary techniques for extending the EKR Theorem.  Whereas many EKR papers involve higher-level algebraic or other advanced techniques, we intentionally stay simple in our tactics.  While we correspondingly cede some generality in our theorems, we note that proofs requiring little machinery have a natural appeal, and it is always a surprise and a pleasure to discover the true potential of elementary methods. While the Erd\H os--Ko--Rado theorem applies specifically to sets, the directions that we will be most concerned with involve finding similar results for other combinatorial objects.  In particular, we will be examining $k$-permutations, words, and multisets, including considering what happens when the standard notion of intersection is replaced by Hamming intersection.

\section{Historical Overview}

The original EKR Theorem appeared at the head of a wave of results in what is now known as extremal set theory.  It was not strictly the first result in the field, appearing after works such as Ramsey's Theorem~\cite{ramseytheory} and the Erd\H os-Szekeres paper on combinatorial geometry~\cite{happyending} as well as taking inspiration from the work of Sperner~\cite{sperner}; however, the EKR paper was certainly a pioneer.  In the literature, there exist many generalizations of the theorem, and we provide only a sampling of them here.  Results immediately relevant to our current study will be mentioned as needed.

One very direct generalization is to think about families that are $t$-intersecting; that is, for $S,T\in\script F$, we have $|S \cap T|\geq t$.  Again, we can think about the ``trivial'' family that contains all the $k$-sets sharing a fixed core of $t$ elements; such a $t$-intersecting family contains $\binom{n-t}{k-t}$ elements.  Indeed, the original EKR paper proved that for $n$ large enough, this is always optimal.  Later, Frankl~\cite{franklt} applied random walks to show that $n$ is ``large enough'' when $n\geq (t + 1)(k - t + 1 )$ given that $t\geq 15$; Wilson~\cite{wilson} then extended this result to all $t$ using linear algebraic techniques.

Others have since determined what happens when $n<(t + 1)(k - t + 1 )$.  These partial results culminated in Ahlswede and Khachatrian's~\cite{completeintersection} Complete Intersection Theorem.  This theorem is as powerful as its name seems to imply, and its proof both utilized and honed high powered techniques that have wide-reaching implications for EKR-type results.

Another vein of generalizations has been in abstracting to objects other than intersecting families of $k$-sets.  Some have looked at allowing $\script F$ to be partitioned into a fixed number of classes which must be intersecting (instead of $\script F$ as a whole); Frankl and F\" uredi~\cite{franklandfuredi} is one example of such a study.  Other works include Katona~\cite{katonaintersection}, which uses elegant techniques to arrive at an intersection theorem for systems of sets; there are also results for partitions, such as Meagher and Moura~\cite{setpartitions} or Ku and Renshaw~\cite{permutationspartitions}, and a variety of other classes of combinatorial objects.

There is also a contingent of EKR papers that seek not to derive new results but instead to provide new proofs for old ones.  Katona's proof is one such paper, but there are a number more.  These include Balogh and Mubayi~\cite{shortproof} as well as Frankl and Tokushige~\cite{frankl-erdoskorado}, amongst others.  In general, there is much to be gained from a fresh perspective and departure from standard techniques.  We would like to emphasize that this in no way detracts from the immensely clever and powerful methods used to obtain these results in the first place; rather, these reproofs serve to broaden our understanding of the relevant results.

\section{Families Under Standard Intersection}

In this section, we will derive Erd\H os--Ko--Rado-like theorems for families of objects under the standard notion of intersection.  Our first modification to the EKR problem will be to add a component of order to our $k$-sets.  So instead of looking at families of $k$-sets, we will look at families of $k$-permutations of $[n]$ (we define a $k$-permutation as an ordered $k$-set).  We'll define a family $\script P$ of $k$-permutations to be \textit{intersecting} if every $P,Q\in\script P$ satisfy $|P\cap Q|\geq 1$.  A trivial family has all of its $k$-permutations sharing a common element (and contains all possible sets possessing said element).

\begin{theorem}
\label{restricted words}
Let $\script P$ be an intersecting collection of $k$-permutations on $[n]$, $n\geq 2k$.  Then
$$|\script P|\leq k!\binom{n-1}{k-1}.$$
Furthermore, equality occurs if and only if $\script P$ is a trivial family.
\end{theorem}

\begin{proof}
We break $\script P$ up into classes that have the same relative ordering of their elements.  Define an equivalence relation $\sim$ on $\script P$ such that for $P=(p_1,\ldots, p_k),Q=(q_1,\ldots, q_k)$, $P\sim Q$ when for all $1\leq i,j\leq k$,
$$p_i \leq p_j \Leftrightarrow q_i \leq q_j$$
(that is, $P$ and $Q$ have the same ordering of their elements).  Then we can apply the Erd\H os--Ko--Rado theorem to find that each equivalence class has at most $\binom{n-1}{k-1}$ elements, and there are clearly at most $k!$ equivalence classes, leading to a total of at most $k!\binom{n-1}{k-1}$ elements.  Furthermore, this maximum is attained only if each equivalence class is a trivial family; it is not hard to show that $\script P$ must also have been a trivial family.  After we check that the trivial family attains the maximum value, the result follows.
\end{proof}

Note that this theorem followed by strategically reducing our new problem to a previous EKR result.  We will see that this is a general theme, and hence a technique of some promise.

We now head in a different vein.  Instead of imposing order on our sets, we'll drop the restriction that elements must appear only once in our sets.  That is, we'll consider intersecting families of $k$-multisets, or multisets with $k$-elements.  The definition of an intersecting family $\script M$ of multisets is precisely analogous to what we've seen before (every pair of elements in $\script M$ have at least one element in common), and a trivial family is the collection of all multisets containing a fixed element of $[n]$.

\begin{theorem}
\label{multisets}
Let $\script M$ be an intersecting collection of $k$-multisets on $[n]$, $n\geq 2k$.  Then
$$|\script M| \leq \binom{n+k-2}{k-1}$$

Equality is attainable and occurs if and only if $\script M$ is a trivial collection; that is, there is some element that belongs to every multiset of $\script M$.
\end{theorem}

\begin{proof}
Let $S_i(\script M)=\{M | M\in\script M \text{ and $M$ contains precisely $i$ distinct elements}\}$.  Since
$$\bigcup_{i=1}^k \script S_i = \script M,$$
we see that $\{\script S_i\}$ defines a partition of $\script M$.  Furthermore, since $\script M$ is intersecting, it follows that each $\script S_i$ is intersecting as well.  If $M$ is a multiset and $S$ is a set, we say that $M$ \textit{reduces to} $S$ if the elements of $M$ and $S$ are precisely the same.  Let $\script S_i'=\{S:\exists M\in \script S_i\text{ such that $M$ reduces to $S$}\}$.  Since $S_i'$ is an intersecting family of $i$-sets, by the Erd\H os-Ko-Rado theorem we have that $|S_i'|\leq\binom{n-1}{i-1}$.

Now let $S \in S_i'$.  The number of $k$-multisets that reduce to $S$ can be calculated as $\binom{k-1}{i-1}$ using your standard stars-and-bars counting argument.  So we have that $|S_i| \leq \binom{k-1}{i-1}\binom{n-1}{i-1}$.  Thus, we have that
$$\begin{array}{lll}
|M| = \sum_{i=1}^{k} |S_i| & \leq & \sum_{i=1}^{k} \binom{k-1}{i-1}\binom{n-1}{i-1}\\
                           & =    & \sum_{i=0}^{k-1} \binom{k-1}{i}\binom{n-1}{n-i-1}\\
                           & =    & \binom{n+k-2}{k-1}
\end{array}$$
after massaging appropriately.

Now if $\script M$ is a trivial family, we can apply the same sort of counting argument to obtain $|\script M| = \binom{n+k-2}{k-1}$, the above maximal value.  Conversely, if $\script M$ is an intersecting family such that $|\script M|$ is maximal, then we must have equality in all of the inequalities we summed.  In particular, $S_1 = \binom{k-1}{0}\binom{n-1}{0} = 1$, meaning that $\script M$ contains a set with only one distinct element.  Since $\script M$ is intersecting, every other element of $\script M$ contains this element as well, implying that $\script M$ is a trivial collection.
\end{proof}

Finally, we make one more transition to round out this section.  We now both add a component of order and remove the restriction that elements must appear only once in our $k$-sets.  That is, we now consider $k$-words on $[n]$, or ordered $k$-tuples such that each entry is an element of $[n]$.  We will keep the same concept of intersection (the intersection $W \cap V$ between two words is defined as the multiset of elements that appear somewhere in both $W$ and $V$, multiplicities included), intersecting families, and trivially intersecting families.

\begin{theorem}
\label{words}
Let $\script W$ be an intersecting collection of $k$-words on $[n]$, $n\geq 2k$.  Then
$$|\script W|\leq n^k - n^{k-1}.$$
Furthermore, $|\script W|$ attains this maximal value when and only when $\script W$ is trivial.
\end{theorem}

\begin{proof}
As before, let's try to relate this version of EKR back to what we already know.  Define an equivalence relation $\bowtie$ on $\script W$ as follows: for $W = w_1 \ldots w_{k}, V = v_1 \ldots v_{k}$, then $W\bowtie V$ when, for all $1\leq i,j\leq k$,
$$w_i = w_j \Leftrightarrow v_i = v_j.$$
Informally, we could say that $W$ and $V$ have the same pattern of equality of letters.  It is not hard to check that indeed $\bowtie$ is an equivalence relation.  If $\script E$ is an equivalence class of $\bowtie$, let $\script E'$ be the set obtained by converting all of $\script E$'s elements to $k$-permutations (in particular, by retaining only the first occurrence of each element).  So if $\script E = \{113233, 223433,\ldots\}$, we would have $\script E' = \{132,234,\ldots\}$.

We note that by construction, $\script{E'}$ is an intersecting family of $i$-permutations for some $i\leq k$, and hence by Theorem \ref{restricted words} has maximal size if it is a trivial family.  Furthermore, $|\script E| = |\script E'|$, so we see that $|\script W|$ is maximized if all of the $|\script E'|$ are maximized.  It is not hard to check that if $\script W$ is a trivial family, each $\script E'$ is a trivial family of $k$-permutations, and hence $|\script W|$ is maximized.  Conversely, if $|\script W|$ is maximal, then there must be an equivalence class containing sets with only one distinct element, implying that $\script W$ is trivial.  We calculate the size of a trivial collection as $n^k - n^{k-1}$, completing our proof.
\end{proof}

We have sampled only a quick bite of the diversity of generalizations available, and we hope we have whetted the reader's appetite for more.  In our next section, we consider a different definition of what it means to be intersecting.

\section{Families Under Hamming Intersection}

A notion from coding theory will shape our work in this next section.  Recall that the Hamming distance between two words is defined as the number of positions in which the two words differ (so the Hamming distance between $(1,3,5)$ and $(1,4,5)$ is 1).  Analogously, we define the size of the Hamming intersection between two words $W$ and $V$, $|W\cap_H V|$, to be the number of positions in which $W$ and $V$ agree.  Thus we have that $|(1,3,5)\cap_H (1,4,5) = 2$.

In this new context, we return to and extend our previous results for $k$-permutations and $k$-words (Hamming intersection is not defined for multisets).  First of all, we should start thinking about $k$-permutations.  The concepts we used before extend readily: a family $\script P$ is \textit{Hamming intersecting} if for each pair $P,Q\in\script P$, $|P\cap_H Q|\geq 1$, and a trivial family is one that has a fixed element appear in a fixed position for each permutation in the family (and contains all such $k$-permutations).

Notice that our proof of Theorem \ref{restricted words} is useless in this new context, and hence we must turn to a new technique.  As it turns out, we can use a variant of Katona's~\cite{katona} probabilistic method.  Ku and Leader~\cite{kuandleader} were the first to notice this, and they successfully developed a proof that involved examining bijections between $[n]\times [n]$ and $[n^2]$.  We present our own variant that is slightly simpler.

\begin{theorem}
\label{hamming restricted words}
Let $\script P$ be a Hamming intersecting family of $k$-permutations on $[n]$, $n\geq k$.  Then
$$|\script P| \leq \frac{(n-1)!}{(n-k)!}.$$
Furthermore, $|\script P| = \frac{(n-1)!}{(n-k)!}$ is attainable by the trivial family.  (However, we make no claim that it is only attainable by such a family.)
\end{theorem}

\begin{proof}
Let $\sigma = \sigma(1) \sigma(2) \ldots \sigma(n)$ be a permutation of $\{1,2,\ldots,n\}$.  We see that when read cyclically, each $\sigma$ contains at most 1 element of $\script P$ as a subword, since any two distinct length $k$ subwords of $\sigma$ will not have any two letters in the same position.  Furthermore, each $k$-permutation in $\script P$ is a subword of exactly $n(n-k)!$ permutations ($n$ choices for the position of the head of the subword, and the other $n-k$ letters can be arranged in any order).  Thus we have that
$$|\script P| n (n-k)! \leq n!\cdot 1,$$
which simplifies to
$$|\script P| \leq \frac{(n-1)!}{(n-k)!}.$$
Furthermore, if we let $\script P$ be a trivial family, we have that
$$|\script P| = (n-1)(n-2)\cdots (n-k+1) = \frac{(n-1)!}{(n-k)!},$$
as desired.
\end{proof}

The remaining natural question is to ask is for the case of $k$-words with Hamming intersection.  At this point, the reader can likely predict how we're going to define intersecting families, but we include it for completeness.  We say that $\script W$ is a \textit{Hamming intersecting family of words} if for each $W,V\in\script W$, we have that $|W\cap_H V|\geq 1$, and a trivial family is one that contains all words having a certain fixed element appearing in a certain fixed position.

Words under Hamming intersection have come up in an EKR context in a variety of papers.  Ahlswede and Khachatrian~\cite{ahlswede98diametric} give an excellent overview of what has been done in this area.  In their paper, Ahlswede and Khachatrian prove an EKR-like theorem for the $t$-intersecting case.  Incidentally, Frankl and Tokushige~\cite{frankl-erdoskorado} came to the same theorem in a different context, utilizing another set of tactics.  While our methods are not powerful enough to attack the $t$-intersecting version, we use the toolkit we have been building in this paper to concisely prove the 1-intersecting case.

\begin{theorem}
\label{hamming words}
Let $W$ be a Hamming intersecting collection of $k$-words on $n$.  Then
$$|W| \leq n^{k-1}.$$
Also, we can obtain $|W| = n^{k-1}$ via a trivial family.
\end{theorem}

\begin{proof}
Recall our equivalence relation $\bowtie$ from the proof of Theorem \ref{words}, where given $W = w_1 \ldots w_{k}, V = v_1 \ldots v_{k}$, $W\bowtie V$ when for all $1\leq i,j\leq k$,
$$w_i = w_j \Leftrightarrow v_i = v_j.$$
As before, let $\script E$ be an equivalence class of $\bowtie$, and define $\script E'$ as the set obtained by converting all of $\script E$'s elements to $k$-permutations.

We note that by construction, $\script E'$ is a Hamming intersecting collection of $i$-permutations for $i\leq n$, and hence by Theorem \ref{hamming restricted words} has maximal size if it is a trivial family.  Furthermore, $|\script E| = |\script E'|$, so we see that $|\script W|$ is maximized if all of the $|\script E'|$ are maximized.  It is not hard to check that if $\script W$ is a trivial collection, each $\script E'$ is a trivial construction of $k$-permutations, and hence $|\script W|$ is maximized.  We calculate the size of a trivial collection as $n^{k-1}$, completing our proof.
\end{proof}

\section{Conclusions and Future Directions}

At this point, we have visited a number of different generalizations of the Erd\H os--Ko--Rado Theorem.  In each of our proofs, we noticed that we required only elementary techniques to arrive at the desired conclusion.  Some of our results have long been discovered, but the methodology is the real gem to mine from this text.

Our results are only the tip of the iceberg, however.  There are many possible directions to go from here.  Some of our proofs generalize immediately to $t$-intersecting systems (where the size of the intersection is required to be at least $t$ instead of at least 1); others may not generalize at all.  We note that our proof of Theorem \ref{hamming words} in particular has potential to be generalized to $t$-Hamming intersecting collections (which would be essentially the same as the Ahlswede and Khachatrian~\cite{ahlswede98diametric}---Frankl and Tokushige~\cite{frankl-erdoskorado} result).  However, our generalization would rely on the following conjecture:

\begin{conjecture}
\label{hamming restricted words conjecture}
Let $\script P$ be a $t$-Hamming intersecting collection (every $P,Q\in\script P$ satisfies $|P\cap_H Q|\geq t$) of $k$-permutations on $[n]$, $n\geq n_0 (k,t)$.  Then
$$|W|\leq \frac{(n-t)!}{(n-k)!}.$$
\end{conjecture}

One should note that the Katona-style probabilistic argument we used in Theorem \ref{hamming restricted words} breaks down for $t$-intersecting families.  However, all hope is not lost.  Recently, there has been been some work done on trying to generalize Katona's proof to the 2- and 3-intersecting cases; see Howard and K\'arolyi~\cite{howard-towards} for details.

Another conjecture we would like to pose is the following:

\begin{conjecture}
Let $\script M$ be a $t$-intersecting collection of $k$-multisets on $[n]$, $n\geq n_0 (k,t)$.  Then
$$|\script M|\leq\binom{n+k-t-1}{k-t}.$$
Furthermore, equality is achieved if and only if $M$ is a trivial collection.
\end{conjecture}

An interesting note is that the conjecture, if true, would imply that there are as many $t$-intersecting $k$-multisets on $[n]$ as there are $t$-intersecting $k$-sets on $[n+k-1]$, given appropriate $n,k,t$.

In any case, we hope that the reader has enjoyed this foray into elementary techniques for Erd\H os--Ko--Rado results.  Our paper is far from the first, and hopefully far from the last, to attempt to find simple proofs for these facts.  We hope that at the very least, we have inspired the reader to consider simple methods as tools for constructing elegant solutions to general problems.

\section{Acknowledgements}

The authors are indebted to Dr. Anant Godbole for his supervision at the 2008 East Tennessee State University REU.  This work was supported by NSF grant 0552730.

\bibliographystyle{amsplain}
\bibliography{ekrbib}

\begin{tabular}{p{3in}l}
\small\textsc{Greg Brockman} & \small\textsc{Bill Kay}\\[-5pt]
\small\textsc{Harvard University} &\small\textsc{University of South Carolina}\\[-5pt]
\small\textsc{Cambridge, MA}&\small\textsc{Columbia, SC}\\[-5pt]
\small\textsc{United States}&\small\textsc{United States}\\[-5pt]
\small \verb|gbrockm@fas.harvard.edu|&\small\verb|kayw@mailbox.sc.edu|
\end{tabular}\\

\end{document}